\documentclass[12pt, reqno]{amsart}
\usepackage{amsmath,amscd,graphics,graphicx,color,a4wide,hyperref,verbatim}
\usepackage{tgpagella}
\usepackage{multicol, fullpage}
\usepackage[usenames,dvipsnames]{xcolor} 
\usepackage{tikz, tikz-3dplot, pgfplots}
\usepackage{tkz-graph}
\usepackage{tikz-cd}
\usetikzlibrary[positioning,patterns] 
\usetikzlibrary{matrix,arrows,decorations.pathmorphing}

\usepackage{textcmds}
\usepackage{epic}
\usepackage{graphicx}
\usepackage{psfrag}
\usepackage{marvosym}
\usepackage{amsmath}
\usepackage{amsfonts}
\usepackage{amssymb}
\usepackage{amsthm}
\usepackage{mathrsfs}
\usepackage{enumitem}
\usepackage[parfill]{parskip} 

\usepackage{subcaption}
\usepackage{ytableau}

\usepackage{multirow}
\usepackage{tikz}
\usetikzlibrary{tikzmark}
\usepackage{hhline}
\usepackage{dynkin-diagrams}

\usepackage{slashed}

\usepackage{comment}

\usepackage{calligra}
\usepackage{mathrsfs}
\DeclareMathAlphabet{\mathcalligra}{T1}{calligra}{m}{n}
\DeclareFontShape{T1}{calligra}{m}{n}{<->s*[1.5]callig15}{}

\newtheorem{theorem}{Theorem}[section]
\newtheorem{lemma}[theorem]{Lemma}
\newtheorem{lem-def}[theorem]{Lemma-definition}
\newtheorem{proposition}[theorem]{Proposition}
\newtheorem{prop-def}[theorem]{Proposition-definition}

\newtheorem{conjecture}[theorem]{Conjecture}

\theoremstyle{definition}

\newtheorem{remark}[theorem]{Remark}
\numberwithin{equation}{section}

\newtheorem{thm}{Theorem}[section] 
\theoremstyle{plain} 
\newcommand{\thistheoremname}{}
\newtheorem{genericthm}[thm]{\thistheoremname}
  
\newtheorem*{genericthm*}{\thistheoremname}
\newenvironment{namedthm*}[1]
  {\renewcommand{\thistheoremname}{#1}%
   \begin{genericthm*}}
  {\end{genericthm*}}

\newcommand{\CC} {\mathbb{C}}

\newcommand{\LL} {\mathbb{L}}

\newcommand{\PP} {\mathbb{P}}

\newcommand{\RR} {\mathbb{R}}

\newcommand {\shA} {\mathcal{A}}
\newcommand {\shB} {\mathcal{B}}

\newcommand {\shD} {\mathcal{D}}
\newcommand {\shE} {\mathcal{E}}

\newcommand {\shG} {\mathcal{G}}

\newcommand {\shM} {\mathcal{M}}
\newcommand {\shN} {\mathcal{N}}
\newcommand {\shO} {\mathcal{O}}
\newcommand {\shQ} {\mathcal{Q}}
\newcommand {\shR} {\mathcal{R}}
\newcommand {\shS} {\mathcal{S}}

\newcommand {\shP} {\mathcal{P}}
\newcommand {\shU} {\mathcal{U}}
\newcommand {\shV} {\mathcal{V}}
\newcommand {\shW} {\mathcal{W}}
\newcommand {\shX} {\mathcal{X}}
\newcommand {\shY} {\mathcal{Y}}

\newcommand{\sExt}{\mathscr{E} \kern -3pt xt}

\newcommand {\Hom} {\operatorname{Hom}}
\newcommand {\sHom}{\mathscr{H}\kern-5pt\mathcalligra{om}}

\newcommand {\Pic} {\operatorname{Pic}}

\newcommand {\Tot} {\operatorname{Tot}}

\newcommand {\arw} {\longrightarrow}

\newcommand{\spinor}{\mathbb S}

\DeclareRobustCommand{\Sec}{\ifmmode\mathsection\else\textsection\fi}

\makeatletter
\newcommand\longleftrightarrowfill@{%
  \arrowfill@\leftarrow\relbar\rightarrow}
\makeatother

\title[]{Derived Equivalences of Generalized
Grassmannian Flops: $D_4$ and $G_2^{\dagger}$
Cases}

\author[Y. Xie]{Ying Xie}
\address{Southern University of Science and Technology \\ 1088 Xueyuan Avenue
 \\ Shenzhen 518055, China.}
\email{xiey@sustech.edu.cn}

\pgfplotsset{compat=1.18}

\begin{document}

\begin{abstract}
    We prove that the generalized Grassmannian flops of both $D_4$ and $G_2^{\dagger}$ type induce derived equivalences, which provide new evidence for the DK conjecture by Bondal-Orlov and Kawamta. The proof is based on Kuznetsov's mutation technique, which takes a sequence of mutations of exceptional objects.   
\end{abstract}
\maketitle

\section{Introduction}
\subsection{Background} Flops represent a class of birational transformations that play a crucial role in the minimal model program (MMP), which occurs when two birational varieties have isomorphic canonical bundles. Derived categories of coherent sheaves encode rich geometric information and allow for subtle comparisons between birational models, which has led to the formulation of DK conjectures about the behavior of derived categories under flops by Bondal-Orlov and Kawamata independently:
\begin{conjecture}[DK conjecture \cite{bondal2002derived, kawamata2002d}]\label{conj:DK_conjecture}
    Consider a birational map $\mu:X_1\arw X_2$ of smooth varieites, which is resolved by morphisms $f_i:X_0\arw X_i$ such that $f_1^*K_{X_1}$ and $f_2^*K_{X_2}$ are linearly equivalent (i.e. a K-equivalence). Then $X_1$ and $X_2$ are derived equivalent.
\end{conjecture}
DK conjecture holds for typical examples of K-equivalences such as standard flops and Mukai flops, as well as some sporadic flops like $G_2$ flops (\cite{kuznetsov2018derived}, \cite{ueda2019g_2}) and $C_2$ flops(\cite{segal2016new}, \cite{morimura2021derived}, \cite{hara2022derived}).  Many more examples can be constructed from the geometry of homogeneous varieties, the so-called \emph{generalized Grassmann flops} (\cite{kanemitsu2022mukai} \cite{cflip}). In this paper, we provide two new classes of generalized Grassmann flops ($D_4$ and $G_2^{\dagger}$) which satisfy Conjecture \ref{conj:DK_conjecture}.

\subsection{Main Results} The geometric construction of the $D_4$ flop and $G_2^{\dagger}$ flop are due to the notion of roof introduced by Kanemitsu \cite{kanemitsu2022mukai}. Fix $V_8=\CC^8$, and consider the orthogonal Grassmannian $OGr(3, V_8)$ whose elements consist of isotropic 3-dimensional subspaces of $V_8$ with respect to the standard quadratic form of $\CC^8$, and $OGr(3, V_8)$ admits two $\PP^3$ fibrations onto the two connected components $OGr(4, V_8))_{\pm}$ of $OGr(4, V_8)$, which is the $D_4$ roof in \cite{kanemitsu2022mukai}. Then the generalized Grassmannian $D_4$-flop is given by a flop $X_+$ to $X_-$ resolved by two blow-ups such that they share the same exceptional divisor isomorphic to $OGr(3, V_8)$. See Section \ref{D4} for more details.  

Fix a smooth 5-dimensional quadric hypersurface $Q$ and any rank 3 Ottaviani bundle $\shG$ on it. The projectiviazation $\PP(\shG)$ admits another $\PP^2$ bundle on another smooth 5-dimensional quadric $Q'$, which is the $G_2^{\dagger}$ roof in Section 5 of \cite{kanemitsu2022mukai}. Then the generalized Grassmannian $G_2^{\dagger}$ flop is given by a flop $Y$ to $Y'$ resolved by two blow-ups such that they share the same exceptional divisor isomorphic to $\PP(\shG)$. See Section \ref{G2} for details. 

\begin{theorem}\label{thm}
    Both the two flops $X_+\dashrightarrow X_-$ and $Y\dashrightarrow Y'$ above satisfy the DK conjecture. That means there are equivalences
    $$D(X_+)\cong D(X_-),\,\,\,D(Y)\cong D(Y').$$
\end{theorem}

Theorem \ref{thm} also holds for relative flops over a smooth projective base $B$(Section \ref{rel} for details). Let $\shW$ be a \emph{quadratic vector bundle} of rank eight on $B$. That means there is a qudratic form $q\in H^0(B, Sym^2 \shW^{\vee})$ on $\shW$ such that the morphism $q: \shW\longrightarrow \shW^{\vee}$ is isomorphic. We also assume that the structure group of $\shW$ lies in $spin(8)$. Then the relative orthogonal Grassmannian $OGr(3, \shW)\longrightarrow B$ is a locally trivial fibration with general fiber isomorphic to $OGr(3, V_8)$, which admits a relative $D_4$ roof structure over the base $B$. Then, the same construction shows that there is a flop $\shX_+\dashrightarrow \shX_-$. Similarly, given a relative $G_2^{\dagger}$ roof $\shR\longrightarrow B$, there is a flop $\shY\dashrightarrow \shY'$. Then, we have the following for relative cases as in Theorem \ref{thm}.
\begin{theorem}[relative cases]\label{relthm}
Both the two flops $\shX_+\dashrightarrow \shX_-$ and $\shY\dashrightarrow \shY'$ satisfy the DK conjecture. Namely, there are equivalences 
$$D(\shX_+)\cong D(\shX_-), \,\,\shD(\shY)\cong D(\shY').$$
\end{theorem}

\subsection{Other related work} Both arguments for proving derived equivalences follow from explicit sequences of mutations of exceptional objects, which should be able to apply to other generalized Grassmannian flop cases. So far, the $C_2$, $G_2$, and $A_4^G$, Mukai flop cases have been settled by this method. Recently, the author and Marco Rampazzo carried out $D_5$ cases \cite{rampazzo2024derived}. However, a general mutation pattern for the $D_n$ flop is still missing. Almost the same time, Wahei Hara solved the $G_2^{\dagger}$ case using the tilting method \cite{hara2024derived}.

\subsection{K3 fibrations}  Consider the restriction of the relative $D_4$ roof $OGr(4, \shW)$ to its any generic $(1,1)$ section $M_{D_4}$, and the restriction of the two fibrations onto $OG(3, \shW)_{\pm}$ are generically $\PP^2$ fibrations and locally trivial $\PP^3$ fibrations over some proper subvarieties $\shM_{\pm}\subset OG(4, \shW)_{\pm}$. Similarly, there is a pair of proper subvarieties $\shN\subset \shQ$ and $\shN'\subset \shQ'$ by considering a generic $(1,1)$ section of the relative $G_2^{\dagger}$ roof $\shR\longrightarrow B$. Here $\shQ$ and $\shQ'$ are locally trivial 5-dimensional quadric fibration over $B$. Then we have the following.
\begin{theorem}\label{K3}
The above subvarieties $\shM_{\pm}$ and $\shN, \shN'$ are both K3 fibrations over $B$ such that the general fiber of each is smooth K3 surface of degree 12. Moreover, there are derived equivalences
$$ D(\shM_+)\cong D(\shM_-),\,\, D(\shN)\cong D(\shN').$$
\end{theorem}
 When the base $B$ is a single point, Theorem \ref{K3} gives a derived equivalence for two degree 12 K3 surfaces $D(M_+)\cong D(M_-)$ inside 5-dimensional quadrics. It was first proved in \cite{ito2020derived} by Mukai's description for K3 surfaces of degree 12 and later by Kapustka and Rampazzo (\cite{kapustka2022mukai}) using the derived Torelli theorem. In loc. cit., the authors also proved the derived equivalences of the K3 fibrations under the assumption that the base $B$ has no odd cohomology. We give a proof of the derived equivalence based on the $D_4$ and $G_2^{\dagger}$ flop regardless of that restriction (Section \ref{K312}).

The rest of this paper is organized as follows. In Section \ref{D4}, the $D_4$ flop is discussed in detail. In Section \ref{pfD4}, we give a detailed proof of the derived equivalence for the $D_4$ flop.In Section \ref{rel}, relative $D_4$, $G_2^{\dagger}$ flops, and their derived equivalences are discussed. In Section \ref{G2}, we construct the $G_2^{\dagger}$ roof from the $D_4$ roof and sketch the proof of the derived equivalences of the $G_2^{\dagger}$ flop. In Section \ref{K312}, pairs of K3 surface fibrations from the $D_4$ and $G_2^{\dagger}$ roof are discussed.

\subsection*{Acknowledgments}
This work was announced at the MIST workshop held at the Chinese University of Hong Kong in April 2024. The author would like to thank Professor Conan Leung for organizing the workshop and for valuable suggestions, Marco Rampazzo for introducing the $G_2^{\dagger}$ case, and Ki-Fung Chan, William Donovan, Wehei Hara, and Zhan Li for many helpful discussions.  

\section{Flops of Type $D_4$}\label{D4}

\subsection{Orthogonal Grassmannians}
Fix a vector space $V_{8}\cong \CC^{8}$ and a nondegenerate quadratic form $q$ on it. For every $k\leq 4$, the orthogonal Grassmannian $OG(k,V_{8})$ is defined by 
$$\{V_k:\,V_k\leq V_8, \dim V_k=k, q|_{V_k}=0\},$$
which is a smooth homogeneous variety of the algebraic group $spin(8,\CC)$. When $k=4$, $OG(4, V_{8})$ is disconnected with precisely two connected components $\spinor_+$ and $\spinor_-$, both of which are isomorphic to each other and called the spinor variety of $spin(8,\CC)$. Let $\Delta_{\pm}$ be the two spinor representation of $spin(8, \CC)$, and there is an embedding $\spinor_{\pm}\hookrightarrow \PP(\Delta_{\pm}).$ Denote by $h_+$ the very ample divisor class of $\spinor_{+}$ under this embedding, and similar for $h_-$.  

Now consider $OGr(3, V_{8})$ and fix an isotropic $V_{3}\in OGr(3, V_{8})$. Observe that there are precisely two ways to extend $V_{3}$ to an isotropic of dimension $4$, which gives two $\PP^{3}$ fibrations $p_{\pm}: OGr(3, V_{8})\longrightarrow \spinor_{\pm}$, both of which admits $\shO(1,1):=p_+^*\shO(h_+)\otimes p_-^*\shO(h_-)$ as their common tautological line bundles. Fano varieties with such structures are called roofs (see \cite{kanemitsu2022mukai}). In loc. cit., $OGr(3, V_{8})$ is called a $D_4$ roof, which corresponds to the following marked $D_4$ Dynkin diagram:
\begin{center}
 \dynkin[edge length=1.5, root radius=.1cm]{D}{**XX}.
\end{center}

\subsection{$D_4$ roof and associated flop} The $D_4$ roof structure can be summarized as follows
\begin{equation}\label{D4roof}
\begin{tikzcd}
                        &  OGr(3, V_{8}) \arrow[dl, swap, "p_+"] \arrow[dr,"p_-"] \\
                         \spinor_+ &  &\spinor_-,
\end{tikzcd}
\end{equation}
and the $\PP^{3}$ fibrations $p_{\pm}$ are given by
\[
OGr(3,V_8)\cong \PP_{\spinor_{\pm}}(\shU_{\pm}^{\vee}),
\]
where $\shU_{\pm}^{\vee}$ is the dual of the tautological rank $4$ vector bundle on $\spinor_{\pm}$ respectively.
Let $\shO(1,1)=p_+^*\shO(h_+)\otimes p_-^*\shO(h_-)$, then
\begin{align}\label{pj}
p_{\pm*}\shO(1,1)=\shU_{\pm}(2h_{\pm}).
\end{align}
The relative Euler sequence for $p_{\pm}$ is 
\begin{equation}\label{releuler}
\begin{tikzcd}
0\arrow[r] & \shO(-1,-1) \arrow[r] & \shU_{\pm}^{\vee}(-2h_{\pm}) \arrow[r] & \shV^{\vee}(-2h_{\pm}) \arrow[r] &0,
\end{tikzcd}
\end{equation}
where $\shV$ denotes the tautological rank 3 bundle on $OGr(3, V_8)$.  
Let $X_{\pm}=\Tot_{\spinor_{\pm}}(\shU_{\pm}^{\vee}(-2h_{\pm}))$. Then, by the construction of generalized Grassmannian flops (Section 4 of \cite{kanemitsu2022mukai} or Section 2 of \cite{cflip}), there is an isomorphism
$$Bl_{\spinor_+}X_+\cong Bl_{\spinor_-}X_-:=X,$$
and we have the flop diagram:
\begin{equation}\label{D4flopdiag}
\begin{tikzcd}
&&E=OGr(3,V_{8}) \arrow[ddll,swap,"p_+"] \arrow[d, hook,"j"]\arrow[ddrr,"p_-"] \\
&&  \arrow[dl,swap,"\pi_+"]  X \arrow[dr,"\pi_-"]\\
  \spinor_+ \arrow[r,hook,"i_+"]& X_+ \arrow[rr,dashed]    & & X_- & \arrow[l, hook',swap, "i_-"]\spinor_- .
\end{tikzcd}
\end{equation}
By Orlov's blow-up formula, we have the two following semiorthogonal decompositions (SOD) on $D(X)$:
\begin{align*}
D(X)&=\langle \pi_+^*D(X_+), j_*p_+^*D(\spinor_+),  j_*(p_+^*D(\spinor_+)\otimes \shO(h_-)), j_*(p_+^*D(\spinor_+)\otimes \shO(2h_-))\rangle\\
&=\langle \pi_-^*D(X_-), j_*p_-^*D(\spinor_-),j_*(p_-^*D(\spinor_-)\otimes \shO(h_+)), j_*(p_-^*D(\spinor_-)\otimes \shO(2h_+))\rangle.
\end{align*}

For simplicity, we may rewrite the two SODs of $D(X)$ in the following forms by omitting the obvious functors like pullbacks and pushforwards:
\begin{align}
D(X)&=\langle D(X_+), D(\spinor_+), D(\spinor_+)(h_-), D(\spinor_+)(2h_-)\rangle\label{sod+}\\
&=\langle D(X_-), D(\spinor_-), D(\spinor_-)(h_+), D(\spinor_-)(2h_+)\rangle.\label{sod-}
\end{align}

The basic strategy for the proof of $D_4$ flop (also $G_2^{\dagger}$) in Theorem \ref{thm} is to mutate SOD(\ref{sod+}) to SOD(\ref{sod-}) by some concrete mutations. This usually works when the SOD of both $\spinor_+$ and $\spinor_-$, and hence the left orthogonal of both ${}^{\perp}D(X_+)$ and ${}^{\perp}D(X_-)$ are given by homogeneous vector bundles. Note that $\spinor_{\pm}$ is a 6-dimensional quadric in $\PP^7$, and $D(\spinor_{\pm})$ admits a full exceptional collection consisting of $\shO$, spnior bundles $\shU_{\pm}^{\vee}$ and their multiple twists by $\shO(h_{\pm})$.  

\begin{proposition}[\cite{kapranov1988derived}]\label{S4+}
    $D(\spinor_+)$ admits a full exceptional collection, which can be represented by the diagram in Figure \ref{S4} (up to multiple twists of $\shO(h_+)$).
    \begin{figure}[ht!]
    \centering
    \includegraphics[width=0.5\linewidth]{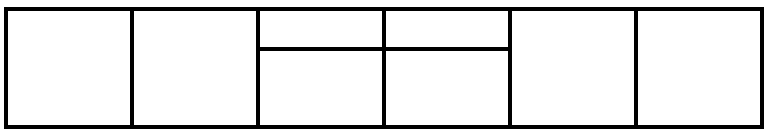}
    \caption{SOD of $\spinor_4^+$}\label{S4}
\end{figure}

Here, the cell \includegraphics[width=0.035\linewidth]{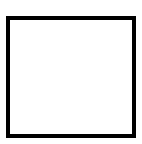} represents $\langle \shO\rangle$
and the cell \includegraphics[width=0.035\linewidth]{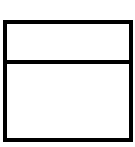} represents $\langle \shO,\shU_+^{\vee}\rangle$.
\end{proposition}
To make the sequences of mutations readable, we borrow the notion of \emph{chessboard} in \cite{jiang2021categorical}, which originates from Kuznetsov \cite{kuznetsov2007homological} and Richard Thomas \cite{thomas2018notes} used in the theory of homological projective duality. 

Making use of Proposition \ref{S4+}, the SOD of ${}^{\perp}D(X_+)$ in (\ref{sod+}) can be represented by a two-dimensional diagram (see Figure \ref{X+}), called chessboard in the following sense:  
\begin{enumerate}
\item There are two types of cells, and each type represents a subcategory generated by some exceptional objects of ${}^{\perp}D(X_+)$.
\begin{itemize}
    \item The cell \includegraphics[width=0.035\linewidth]{attachments/O.png} represents $\langle \shO_E\rangle$
    \item The cell \includegraphics[width=0.035\linewidth]{attachments/OS.png} represents $\langle \shO_E, \shU_+^{\vee}\rangle$
\end{itemize}
\item A pair of integers $(a,b)$ (the coordinate) attached to each cell means tensoring the cell with $\shO_E(a, b):=\shO_E(ah_++bh_-)$. The plane coordinate system used here is the regular one, i.e., with an increase $h_+$-value from left to right and an increase $h_-$-value upwards. Clearly, the coordinates of cells are determined once any one of them is prescribed. 
\end{enumerate}
\begin{figure}[h!]
    \centering
    \includegraphics[width=0.5\linewidth]{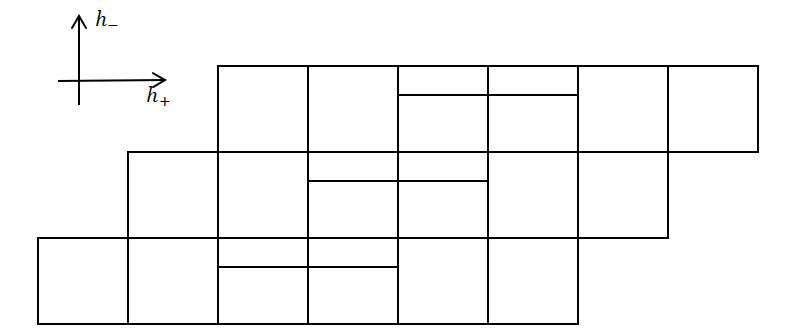}
    \caption{SOD of ${}^{\perp}D(X_+)$}\label{X+}
\end{figure}

SOD of ${}^{\perp}D(X_-)$ in (\ref{sod-}) can be represented by a similar chessboard in Figure \ref{X-}. 
\subsection{Triality} There are three $D_4$ roofs with the following markings:
\begin{center}
\dynkin[edge length=1.5, root radius=.1cm]{D}{**XX}
 \dynkin[edge length=1.5, root radius=.1cm]{D}{X*X*}
 \dynkin[edge length=1.5, root radius=.1cm]{D}{X**X} 
\end{center}
which are isomorphic to each other by triality of the $D_4$ diagram. The last two roofs can be summarized as
\begin{equation}\label{D4+-}
\begin{tikzcd}
                        & OFl(1, 4, V_8)_{\pm}\arrow[dl, swap, "q_{1\pm}"] \arrow[dr,"q_+"] \\
                         OGr(1, V_8):=Q_6 &  & OGr(4, V_8)_{\pm}=\spinor_{\pm}.
\end{tikzcd}
\end{equation}
Here, $OFl(1, 4, V_8)_{\pm}$ denotes the connected component of the isotropic partial $(1,4)$-flag variety in $V_8$ covering $\spinor_{\pm}$ respectively. Denote $h_1$ the hyperplane class of $Q_6$, and we have
$$q_{\pm*}\shO(h_1+h_{\pm})=\shU_{\pm}^{\vee}(h_{\pm}), \, \,q_{1\pm*}\shO(h_1+h_{\pm})=\shS_{\pm}(h_1).$$
$\shS_+$ and $\shS_-$ are the two spinor bundles on $Q_6$. The projective bundle structures are
$$q_{\pm}: \PP_{\spinor_{\pm}}(\shU_{\pm})\longrightarrow \spinor_{\pm},\,\,
q_{1\pm}: \PP_{Q_6}(\shS_{\pm}^{\vee})\longrightarrow Q_6.$$ 
Hence, there are three isomorphic $D_4$ flops:
\begin{align}
 & \Tot_{\spinor_{+}}(\shU_{+}^{\vee}(-2h_{+}))\dashrightarrow \Tot_{\spinor_{-}}(\shU_{-}^{\vee}(-2h_{-})) \label{D4flop}\\
& \Tot_{Q_6}(\shS_{+}^{\vee}(-h_1))\dashrightarrow \Tot_{\spinor_{+}}(\shU_{+}(-h_{+}))\label{D4+flop}\\
& \Tot_{\spinor_{+}}(\shU_{-}(-h_{-}))\dashrightarrow\Tot_{Q_6}(\shS_{-}^{\vee}(-h_1)).\label{D4-flop}
\end{align}
Then Theorem \ref{thm} will give derived equivalences for all these three flops. 

\subsection{Some cohomological Results}
In the rest of this section, we give some cohomological results that are needed in the proof of Theorem \ref{thm}.
\begin{lemma} \label{van}
On the resolution space $X$, we have the following 
\begin{enumerate}
   \item $\RR\Hom (\shO_E(a, -1),\shO_E)=0$ if $a=1, 2, 3, 4$;
   \item $\RR\Hom(\shO_E(3,-2), \shO_E)=0$;
   \item $\RR\Hom(\shO_E,\shU_+^{\vee}(-2, 1))=0$;
   \item $\RR\Hom(\shO_E, \shU_+(-2,1))=0$;
   \item $\RR\Hom(\shO_E, \shU_+^{\vee}(-1,1))=\CC[0]$.
   \item $\LL_{\shO_E(1,-1)}\shU_+^{\vee}=\shV^{\vee}$
   \end{enumerate}
   Note that we omit $j_*$ in the above equalities. 
\end{lemma}
\begin{proof}
For any $A, B\in D(E)$, we have the following by adjunction
$$\RR \Hom_X(j_*A, j_*B)=\RR \Hom_E(j^*j_*A, B).$$
For the closed embedding $j: E\hookrightarrow X$, there is an adjunction triangle
$$ -\otimes \shO(1,1)[1]\rightarrow j^*j_*\rightarrow id \xrightarrow{[1]}.$$
Then, there is a distinguished triangle
\begin{align}\label{dis}
 \RR \Hom_E(A, B)\rightarrow \RR \Hom_E(j^*j_*A, B)\rightarrow \RR \Hom_E (A(1,1)[1],B))\xrightarrow{[1]}.   
\end{align}
To show (1) holds, we are sufficient to show 
\begin{enumerate}
    \item[(i)] $\RR \Hom_E(\shO(a,-1), \shO)=0$ 
    \item[(ii)] $\RR \Hom_E(\shO(a+1, 0),\shO)=0.$
\end{enumerate}
By (\ref{pj}), (i) is equivalent to $\RR\Hom_{\spinor_+}(\shO(a-1), \shU_+)=\RR\Hom_{\spinor_+}(\shU_+(a-1),\shO)=0,$ which holds when $1\leq a\leq 6$ by SOD of $D(\spinor_+)$ (Proposition \ref{S4+}). (ii) is equivalent to $\RR\Hom_{\spinor_+}(\shO(a+1),\shO)=0$ which holds when $0\leq a\leq 4$ also by SOD of $D(\spinor_+)$. Hence, (1) follows. 

Similarly, (2) holds if both $\RR\Hom_E(\shO(3, -2), \shO)=0$ and $\RR\Hom_E(\shO(4,-1),\shO)=0$ hold by the distinguished triangle (\ref{dis}). Note that $\RR\Hom_E(\shO(3, -2), \shO)=\RR \Hom_E(\shO(3,3), p_-^*\shO(5h_-))=0$ holds by the Beilinson decomposition of the projective bundle $p_-: E\rightarrow \spinor_-$, and $\RR\Hom_E(\shO(4,-1),\shO)=0$ holds by (i) above.

(3) holds if both $\RR \Hom_E(\shO, \shU_+^{\vee}(-2,1))=0$ and $\RR\Hom_E(\shO(1,1), \shU_+^{\vee}(-2,1))=0$ hold by the distinguished triangle (\ref{dis}). Note that $\RR \Hom_E(\shO, \shU_+^{\vee}(-2,1))=\RR\Hom_{\spinor_-}(\shU_+^{\vee}, \shU_+^{\vee}(-h_+))=0$ holds by SOD of $D(\spinor_+)$, and $\RR\Hom_E(\shO(1,1), \shU_+^{\vee}(-2,1))=\RR\Hom_{\spinor_+}(\shO(3h_+),\shU_+^{\vee})=0$ holds also by SOD of $D(\spinor_+)$. 

Notice that the Euler sequence of $\spinor_+\subset Gr(4, V_8)$ is
$$0\longrightarrow \shU_+\longrightarrow V_8^{\vee}\otimes \shO\longrightarrow \shU_+^{\vee}\longrightarrow 0.$$ 
Then (4) holds by (1) and (3). 

(5) holds if $\RR\Hom_E(\shO, \shU_+^{\vee})=\CC[0]$ and $\RR\Hom_E(\shO(1,1), \shU_+^{\vee}(-1,1))=0$, which of both holds by SOD of $D(\spinor_+)$. 

(6) follows from (5) and the relative Euler sequence (\ref{releuler}). 
\end{proof}
\section{Proof Theorem \ref{thm}: $D_4$ Case}\label{pfD4}
In this section, we give a complete proof of the derived equivalence of the $D_4$ flop. We start with the SOD of ${}^{\perp}D(X_+)$ represented by the chessboard in Figure \ref{X+}, where the rightmost corner is attached by $(5,2)$.

\begin{description}
\item[Step 1(Figure \ref{M1})] Move the cells at $(4,2)$ and $(5,2)$ to the position at $(1,-1)$ and $(2,-1)$ respectively, and subsequently move the cells at $(-2,0)$ and $(-1,0)$ to the position at $(1,3)$ and $(2,3)$ respectively. 

This can be done by first mutating $\langle \shO_E(4,2),\shO_E(5,2)\rangle$, the rightmost subcategories of SOD (\ref{sod+}), to the far left, resulting in the leftmost subcategories 
$$S_{X}(\langle \shO_E(4,2),\shO_E(5,2)\rangle)= \langle \shO_E(1,-1),\shO_E(2,-1)\rangle.$$
Here, $S_X\cong -\otimes\shO(3E)[\dim X]$ is the Serre functor of $D(X)$. Then left mutate
$D(X_+)$ through $\langle \shO_E(1,-1),\shO_E(2,-1)\rangle$ to the far left, resulting in 
\begin{align}
  \LL_{\langle \shO_E(1,-1),\shO_E(2,-1)\rangle} D(X_+)=:\shD_1. \label{D1} 
\end{align}
Subsequently, exchange $\langle\shO_E(-2,0), \shO_E(-1,0)\rangle$ and $\langle \shO_E(1,-1),\shO_E(2,-1)\rangle$ by means of Lemma \ref{van} (1), and right mutate $\shD_1$ through $\langle\shO_E(-2,0), \shO_E(-1,0)\rangle$, and finally right mutate $\langle\shO_E(-2,0), \shO_E(-1,0)\rangle$ to the far right, resulting in the upmost subcategories $\langle\shO_E(1,3), \shO_E(2,3)\rangle$. 
\begin{figure}[h!]
    \centering
    \includegraphics[width=0.5\linewidth]{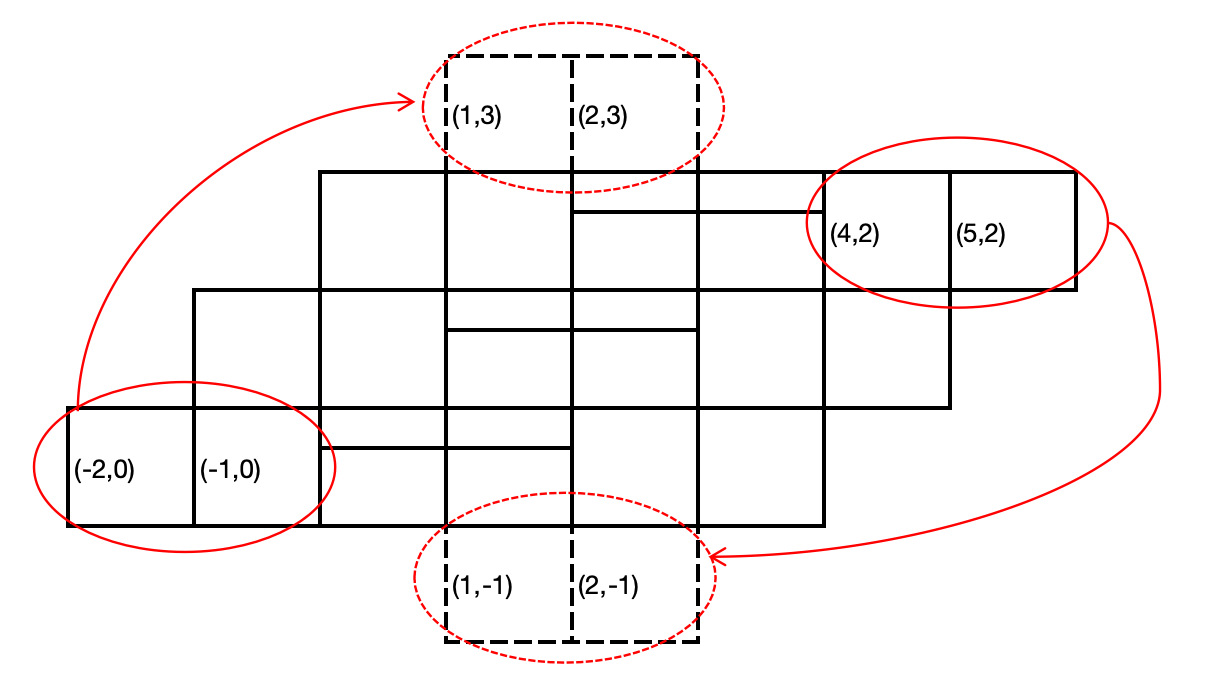}
    \caption{SOD of $^{\perp}\shD_2$}
\end{figure}\label{M1}

Note that the mutated chessboard in Figure \ref{M1} represents the SOD of ${}^{\perp}\shD_2,$ where 
\begin{align}
  \shD_2=\RR_{\langle \shO_E(-2,0), \shO_E(-1,0)\rangle} \shD_1.\label{D2}  
\end{align}

\item[Step 2(Figure \ref{M2})] Insert cells along blue and red arrows as in Figure \ref{M2}. 

We first insert the cells at $(-1,1)$ and $(2,-1)$ into the ones at $(0,0)$ and $(1,0)$ respectively. This can be done by exchanging the cells at $(2,-1)$ and the one at $(0,0)$ by Lemma \ref{van} (1) and (3). We next insert the cells at $(-1,1)$ and $(0,1)$ into the ones at $(0,0)$ and $(1,0)$ respectively. This can be done by first exchanging the cell at $(-1,1)$ and three cells together at $(1,0)$, $(2,0)$ and $(3,0)$ by means of Lemma \ref{van} (1),(2) and (4), and subsequently exchanging the cells at $(0,1)$ and two cells at $(2,0), (3,0)$ by Lemma \ref{van} (1) again. Thereafter, the cell at $(0,0)$ becomes 
\[
\includegraphics[width=0.05\linewidth]{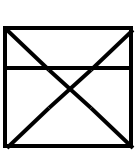}:=\shA=\langle \shO_E(1,-1), \shO_E, \shU_+^{\vee}, \shO_E(-1,1)\rangle,
\]
and the cell at $(1,0)$ becomes $\shA(1,0)$. 
We can exchange $\shO_E(1,-1)$ and $\shO_E$, and left mutate $\shU_+^{\vee}$ through $\shO_E(1,-1)$ in $\shA$ by Lemma \ref{van} (1), (6). Then $\shA$ has the following SOD:
\begin{align}
   \shA=\langle \shO_E, \shV^{\vee}, \shO_E(1,-1), \shO_E(-1,1)\rangle \label{A}
\end{align}
Similarly, other insertions in Figure \ref{M2} follow the same argument as the previous.
\begin{figure}[ht!]
    \centering
    \includegraphics[width=0.5\linewidth]{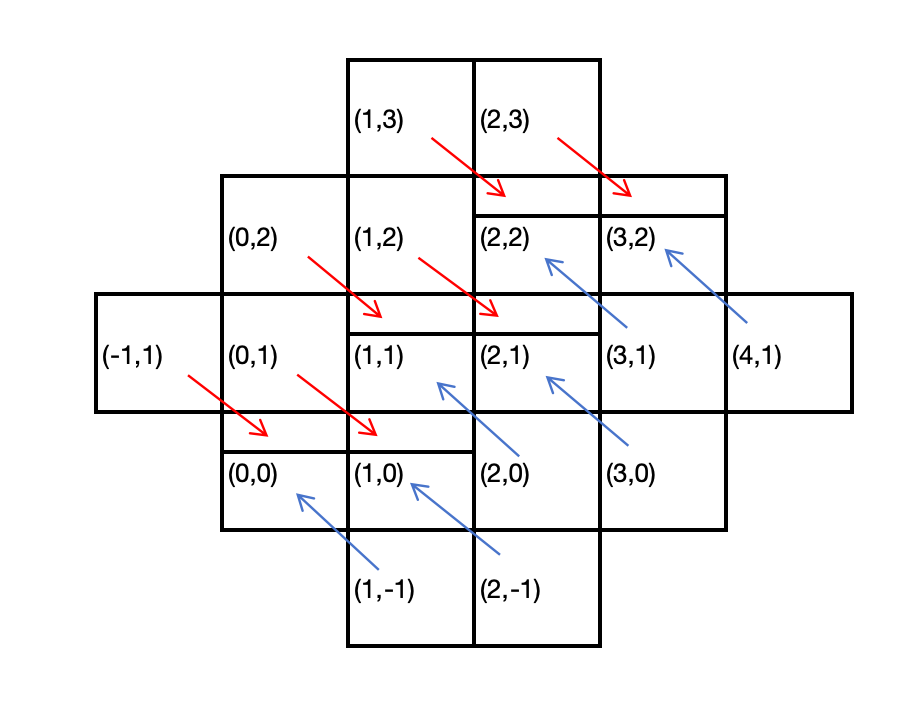}
    \caption{Mutations in $^{\perp}\shD_2$}\label{M2}
\end{figure}

\item[Step 3(Figure \ref{M3})] Left mutate the cells at $(3,2)$ to the far left and left mutate $\shD_2$ to the far left as in Step 1. 
\begin{figure}[h!]
    \centering
    \includegraphics[width=0.5\linewidth]{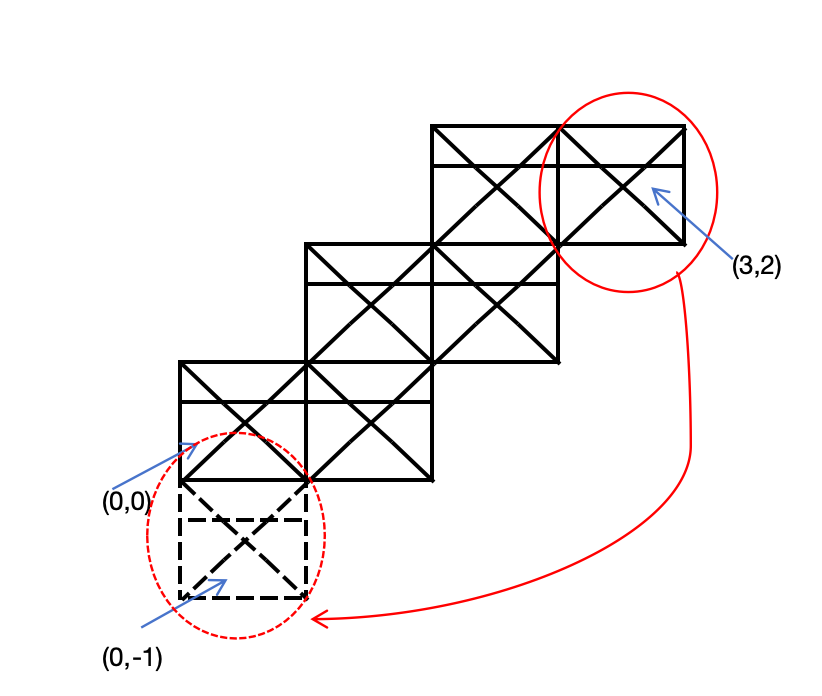}
    \caption{SOD of $^{\perp}\shD_3$}\label{M3}
\end{figure}
Note that the chessboard becomes ${}^{\perp}\shD_3,$ and there is a SOD of $D(X)$
\begin{align}
    D(X)=\langle \shD_3, \shA(0,-1), \shA, \shA(1,0), \shA(1,1), \shA(2,1), \shA(2,2)\rangle,\label{sod++}
\end{align}
where 
\begin{align}
   \shD_3=\LL_{\shA(0,-1)}\shD_2.\label{D3} 
\end{align}

\item[Step 4(Figure \ref{M1-2})] Consider SOD of $^{\perp}D(X_-)$ in (\ref{sod-}) represented by chessboard in Figure \ref{X-}, where the rightmost cell is attached by $(2,4)$. By symmetry of SOD (\ref{sod+}) and (\ref{sod-}), we apply Step 1 and Step 2 to the chessboard in Figure \ref{X-}. 
\begin{figure}[h!]
    \centering
    \includegraphics[width=0.25\linewidth]{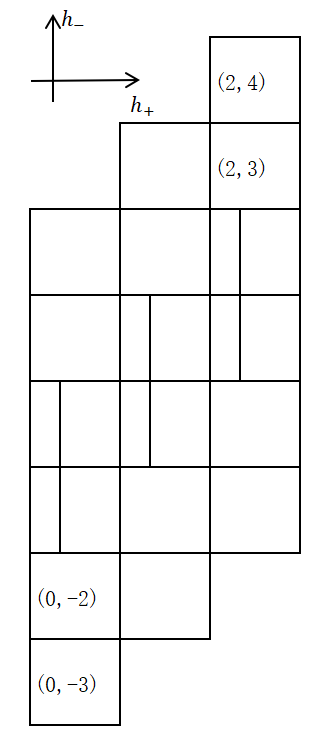}
    \caption{SOD of $^{\perp}D(X_-)$}\label{X-}
\end{figure}
The process is sketched in Figure \ref{M1-2}.
\begin{figure*}[t!]
    \centering
    \begin{subfigure}[t]{0.5\textwidth}
        \centering
        \includegraphics[height=2in]{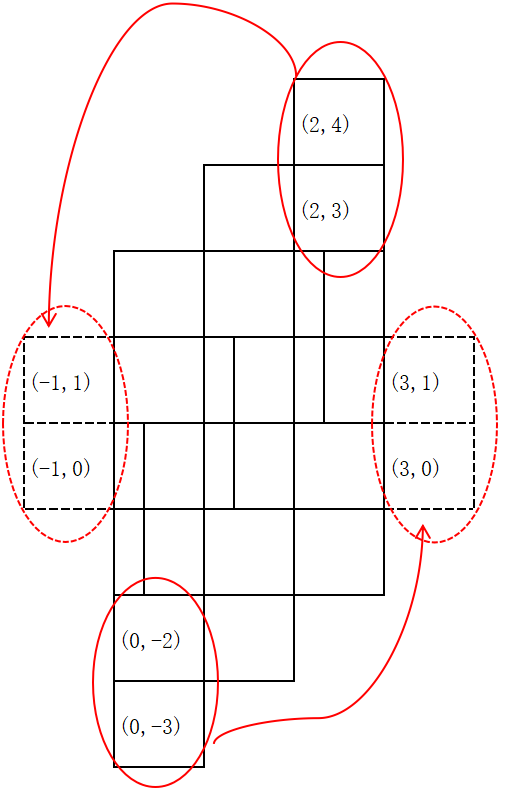}
    \end{subfigure}%
    ~ 
    \begin{subfigure}[t]{0.5\textwidth}
        \centering
        \includegraphics[height=2in]{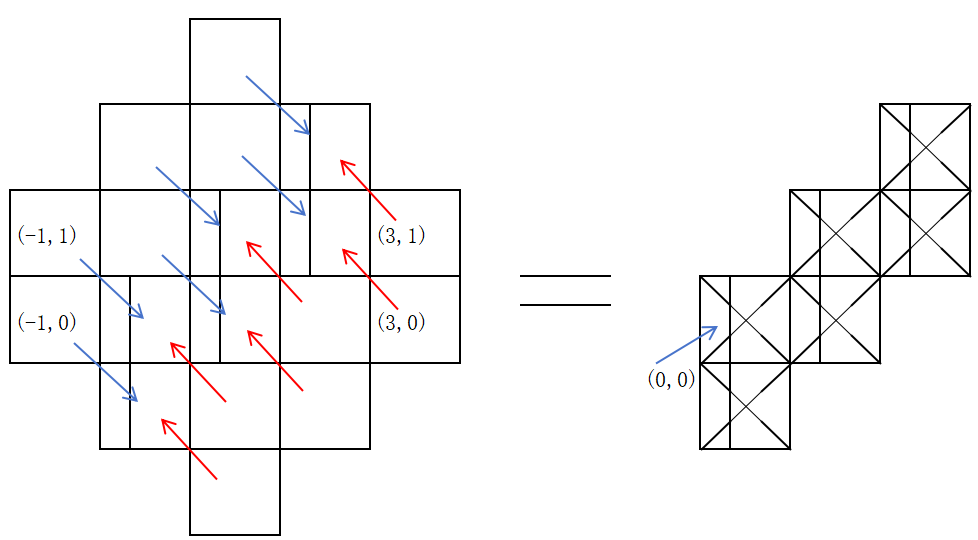}
    \end{subfigure}
    \caption{Mutation of SOD(\ref{sod-})}\label{M1-2}
\end{figure*}
Here, the cell at $(0,0)$ in Figure \ref{M1-2} finally becomes
\[
\includegraphics[width=0.05\linewidth]{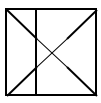}:=\shB=\langle \shO_E(-1,1), \shO_E, \shU_-^{\vee}, \shO_E(1,-1)\rangle,
\]
which also admits a mutation
\begin{align}
    \shB=\langle \shO_E,\shV^{\vee},\shO_E(-1,1),  \shO_E(1,-1) \rangle.\label{B}
\end{align}
Notice that the chessboard in Figure \ref{M1-2} represents SOD of $^{\perp}\shD_-$, and SOD (\ref{sod-}) becomes
\begin{align}
   D(X)=\langle \shD_-, \shB(0,-1), \shB, \shB(1,0), \shB(1,1), \shB(2,1), \shB(2,2)\rangle,\label{sod--}
\end{align}
where
\begin{align}   
\shD_-=\RR_{\langle\shO_E(0,-3),\shO_E(0,-2)\rangle}
\LL_{\langle\shO_E(-1,0), \shO_E(-1,1)\rangle} D(X_-).\label{D-}
\end{align}

\item[Step 5] Comparison of SOD (\ref{sod++}) and SOD (\ref{sod--}). Note that $\shA=\shB$ by (\ref{A}) and (\ref{B}). Hence, there is an equivalence of $$\shD_3\cong \shD_-.$$ 
By the mutations (\ref{D1}), (\ref{D2}), (\ref{D3}) and (\ref{D-}), there is an equivalence
\[
\Phi: D(X_+)\cong D(X_-),
\]
where $\Phi$ is given by
$$
\pi_{-*}\circ\RR_{\langle\shO_E(-1,0), \shO_E(-1,1)\rangle} \circ\LL_{\langle\shO_E(0,-3),\shO_E(0,-2)\rangle} \circ\LL_{\shA(0,-1)} \circ \RR_{\langle \shO_E(-2,0), \shO_E(-1,0)\rangle}\circ \LL_{\langle\shO_E(1,-1),\shO_E(2,-1)\rangle}\circ\pi_+^*.$$
\end{description}

\section{$G_2^{\dagger}$ roof and flop}\label{G2}
\subsection{Ottivani Bundles and $G_2^{\dagger}$ roof}
Let $V_7\subset V_8$ be any hyperplane such that the restriction of the non-degenerate quadratic form $q$ onto $H$ is still non-degenerate. Then, there is an isomorphism between the spinor variety and the maximal isotropic Grassmannian in $H$:
\begin{align}
  & \spinor_{\pm}\xrightarrow{\cong}OGr(3, H)\\\nonumber
  &   V_4 \mapsto V_4\cap V_7.
\end{align}
 This implies that there is a rank 3 vector bundle $\shG_{\pm}$ on $\spinor_{\pm}$, which corresponds to the tautological bundle on $OGr(3, H)$, satisfying the following short exact sequence:
\begin{equation}\label{Ott+-}
\begin{tikzcd}
0\arrow[r] & \shO \arrow[r,"s_{\pm}"] & \shU_{\pm}^{\vee} \arrow[r] & \shG_{\pm}^{\vee} \arrow[r] &0.
\end{tikzcd}
\end{equation}
Note that the section $s_{\pm}: \shO\rightarrow \shU_{\pm}^{\vee}$ in (\ref{Ott+-}) corresponds to the choice of $H\subset V_8$ above.
By triality of $D_4$, there are two short exact sequences on $OGr(1, V_8):=Q_6$:
\begin{equation}\label{Ott1}
\begin{tikzcd}
0\arrow[r] & \shO \arrow[r, "s_{1\pm}"] & \shS_{\pm} \arrow[r] & \shG_{1\pm}^{\vee} \arrow[r] &0.
\end{tikzcd}
\end{equation}
Here $\shS_{\pm}$ denote the spinor bundles on $Q_6$ defined by $q_{1*}q_{\pm}^*(\shO(h_{\pm}))$, respectively, and the section $s_{1\pm}$ correspond to the choice of suitable hyperplanes of the spinor representations $\Delta_{\pm}$.

The total space of projective bundle $\PP(\shG_{\pm})\subset \PP(\shU_{\pm})$ is isomorphic to $OFl(1,3,H)$, the isotropic partial $(1,3)$-flag variety in $H$, which admits a projective bundle on the 5-dimensional quadric $OGr(1, H)$:
\begin{equation}\label{B3}
\begin{tikzcd}
                        & OFl(1, 3, V_7)\cong \PP_{\spinor_{\pm}}(\shG_{\pm}) \arrow[dl, swap, "p_1"] \arrow[dr,"p_3"] \\
                         OGr(1, V_7):=Q &  & OGr(3, V_7)\cong \spinor_{\pm}.
\end{tikzcd}
\end{equation}
Note that diagram (\ref{B3}) is nothing but the restriction of diagram (\ref{D4+-}) to $\PP(\shG_{\pm})\subset \PP(\shU_{\pm})$. The projective bundle $p_1$ is isomorphic to $\PP(\shS^{\vee})$, where $\shS\cong \shS_{\pm}|_Q$ is the unique spinor bundle on $Q$.
The short exact sequences (\ref{Ott1}) restricts to the same one on $Q$:
\begin{equation}\label{Ott}
\begin{tikzcd}
0\arrow[r] & \shO \arrow[r, "s"] & \shS \arrow[r] & \shG^{\vee} \arrow[r] &0.
\end{tikzcd}
\end{equation}
We can further restrict diagram (\ref{B3}) to the total space of projective bundle $\PP(\shG)\subset \PP(\shS^{\vee})$, which admits a projective bundle on a smooth hyperplane section $Q'\subset OGr(3, H)$ by our choice of $s$:
\begin{equation}\label{G2+}
\begin{tikzcd}
                        &  R:=\PP(\shG)  \arrow[dl, swap, "p"] \arrow[dr,"p'"] \\
                         Q &  & Q'.
\end{tikzcd}
\end{equation}
Notice that the projective bundle $p'$ is given by $\PP(\shG_+|_{Q'})\cong \PP(\shG)$.
In fact, $R:=\PP(\shG)$ admits a non-homogeneous $G_2$ action (See \cite{kanemitsu2019extremal} and \cite{ottaviani1990cayley}), and the diagram (\ref{G2+}) is then called a $G_2^{\dagger}$ roof.

\begin{remark}
\begin{enumerate}
    \item It is shown in \cite{ottaviani1988spinor} that 
the bundle $\shG^{\vee}$ on the 5-dimensional quadric $Q$ is stable with Chern class $(c_1, c_2, c_3)=(2,2,2)$. Conversely, in loc. cit., any rank three stable bundle with Chern class $(2,2,2)$ on $Q$ (the so-called \emph{Ottaviani bundle}) arises from some short exact sequence of the form ($\ref{Ott}$).
\item The fine moduli space of Ottaviani bundles on $Q$ is naturally isomorphic to $\PP(H^0(Q, \shS))-\spinor^{\vee}$, where $\spinor^{\vee}\subset \PP(H^0(Q, \shS))$ is the projective dual quadric of the spinor variety $\spinor \subset \PP(H^0(Q, \shS)^{\vee})$. 
\item Any Ottaviani bundle on 5-dimensional $Q$ induces a $G_2^{\dagger}$ roof (\cite{ottaviani1988spinor}). 
\end{enumerate}
\end{remark}

\subsection{$G_2^{\dagger}$ flop}
Let $h$ (resp. $h'$) be the ample generator of $\Pic(Q)$ (resp. $\Pic(Q')$) in diagram (\ref{G2+}). Then 
$$p_*\shO(h+h')=\shG^{\vee}(h), \,\, p'_*\shO(h+h')=\shG'^{\vee}(h').$$
Here, $\shG'$ denotes the dual of the corresponding Ottaviani bundle on $Q'$. Then there is a flop ($G_2^{\dagger}$ flop)
\begin{align}\label{G2dagger}
  Y=\Tot_Q(\shG(-h))\dashrightarrow Y'=\Tot_{Q'}(\shG'(-h')),   
\end{align}
which has a common resolution $\Tilde{Y}$ isomorphic to both $Bl_Q(Y)$ and $Bl_{Q'}(Y')$:
\begin{equation}\label{G2+flop}
    \begin{tikzcd}
&&\PP(\shG^{\vee})\cong \PP(\shG'^{\vee}) \arrow[ddll,swap,"p"] \arrow[d, hook,"j"]\arrow[ddrr,"p'"] \\
&&  \arrow[dl,swap,"\pi"]  \Tilde{Y} \arrow[dr,"\pi'"]\\
  Q \arrow[r,hook,"i"]& Y \arrow[rr,dashed]    & & Y' & \arrow[l, hook',swap, "i'"] Q'. 
\end{tikzcd}
\end{equation}
\subsection{$G_2^{\dagger}$ vs $D_4$} From the construction of $G_2^{\dagger}$ flop, one can deduce that the $G_2^{\dagger}$ flop (\ref{G2dagger}) is a strict transformation of the $D_4$ flop (\ref{D4+flop}). 
\begin{equation}\label{G2D4}
    \begin{tikzcd}
        \Tot_{Q_6}(\shS_{+}^{\vee}(-h_1))\arrow[rr,dashed] &&\Tot_{\spinor_{+}}(\shU_{+}(-h_{+}))\\
        Y=\Tot_Q(\shG(-h))\arrow[u,hook]\arrow[rr,dashed] && Y'=\Tot_{Q'}(\shG'(-h'))\arrow[u,hook]
    \end{tikzcd}
\end{equation}
By triality of $D_4$, the relative Euler sequence (\ref{releuler}) on $OGr(3, V_8)$ induces two relative Euler sequences on $OFl(1, 4, V_8)\cong OGr(3, \Delta_-)$, and hence restricts to two short exact sequences on $\PP({\shG})=R$:  \begin{align}
 & 0\longrightarrow \shO(1, -1) \longrightarrow  \shS \longrightarrow \shE \longrightarrow 0 \label{seqfund1}\\
& 0\longrightarrow \shO(-1, 1) \longrightarrow \shS' \longrightarrow \shE\longrightarrow 0 \label{seqfund2}
\end{align}
Here $\shE$ is the rank three vector bundle on $R$ that corresponds to $\shV_3^{\vee}$ on $OGr(3, \Delta_-)$. 
We also need similar cohomological results as stated in Lemma \ref{van}. 
\begin{lemma}\label{van2}
On common resolution $\Tilde{Y}$, we have
    \begin{enumerate}
        \item $\RR \Hom(\shO_{R}(a,-1), \shO_{R})=0$ if $a=1,2,3$
        \item $\LL_{\shO_{R}(1,-1)}\shS=\shE$
    \end{enumerate}
\end{lemma}
The proof for Lemma \ref{van2} is similar to Lemma \ref{van}. 
\subsection{The proof of Theorem \ref{thm}: $G_2^{\dagger}$ case}\label{pfg2}
There are two SODs of $D(\Tilde{Y})$ by Orlov's blow-up formula:
\begin{align}
&D(\Tilde{Y})=\langle D(Y), D(Q), D(Q)(h')\rangle \label{sod1}\\
&D(\Tilde{Y})=\langle D(Y'), D(Q'), D(Q')(h)\rangle \label{sod2}
\end{align}
Here, the natural pullback and pushforward functors are omitted for simplicity, as in the $D_4$ case. The 5-dimensional quadric $Q$ also admits a full exceptional collection. 
\begin{proposition}\cite{kapranov1988derived}
    D(Q) admits a full exceptional collection represented by the following chessboard (up to multiple twists of $\shO(h)$):
    \begin{figure}[h!]
    \centering
    \includegraphics[width=0.5\linewidth]{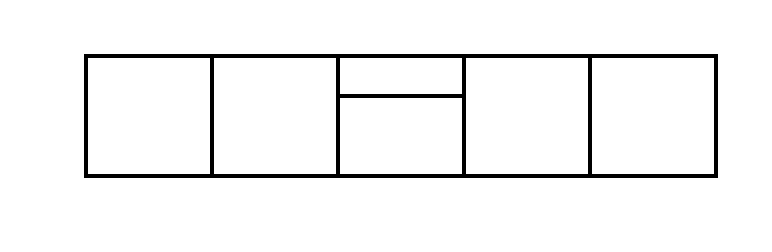}
\end{figure}

Here, the cell \includegraphics[width=0.035\linewidth]{attachments/O.png} represents $\langle \shO\rangle$
and the cell \includegraphics[width=0.035\linewidth]{attachments/OS.png} represents $\langle \shO,\shS\rangle$.
\end{proposition} 
We begin with the chessboard $^{\perp}D(Y)$ in (\ref{sod1}), where the rightmost cell is attached $(3,1)$ shown in Figure $\ref{Y}$.  
 \begin{figure}[ht!]
    \centering
    \includegraphics[width=0.5\linewidth]{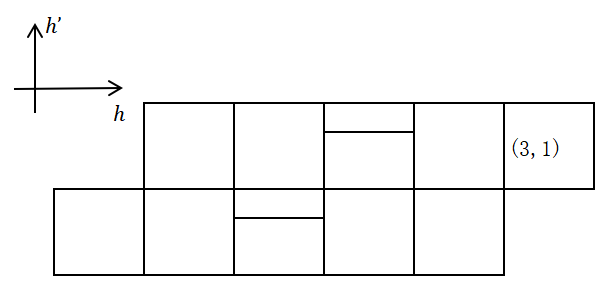}
    \caption{SOD of $^{\perp}D(Y)$}\label{Y}
\end{figure}
\begin{description}
  \item[Step 1](Figure \ref{GS1}) We move the cell at $(3,1)$ to the far left position at $(1,-1)$, and move the cell at $(-2,0)$ to the leftmost position at $(0,2)$ by a similar process as in Step 1 of the $D_4$ case. The Serre functor $S_{\Tilde{Y}}\cong -\otimes\shO(2R)[\dim \Tilde{Y}]$ is applied and Lemma \ref{van2} (1) is used. Note that the right orthogonal of the chessboard (\ref{GS1}) becomes $$\RR_{\shO_{R}(-2,0)}\LL_{\shO_{R}(1,-1)}D(Y):=D_1.$$

   \begin{figure}[ht!]
    \centering
    \includegraphics[width=0.5\linewidth]{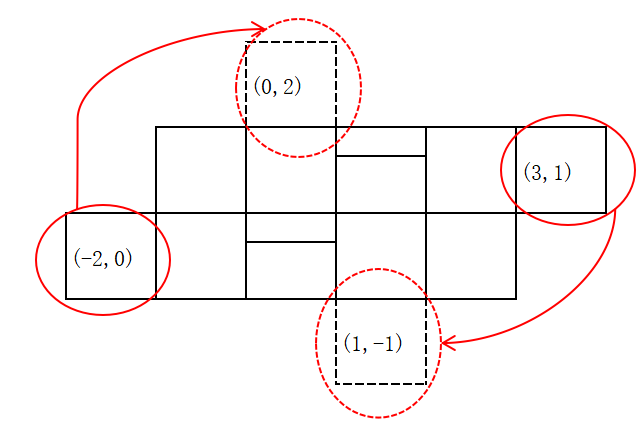}
    \caption{Step 1}\label{GS1}
\end{figure}

  \item[Step 2] (Figure \ref{GM2}) Insert cells at $(1,-1)$ and $(-1,1)$ in to the cells at $(0,0)$, which leads to a new cell $\shA$ at $(0,0)$:
\[
\includegraphics[width=0.05\linewidth]{attachments/B1.png}:=\shA=\langle \shO_{R}(1,-1), \shO_{R}, \shS, \shO_{R}(-1,1)\rangle,
\]
The insertion of $(0,2)$ and $(2,0)$ into $(1,1)$ is similar.  

 \begin{figure}[ht!]
    \centering
    \includegraphics[width=0.5\linewidth]{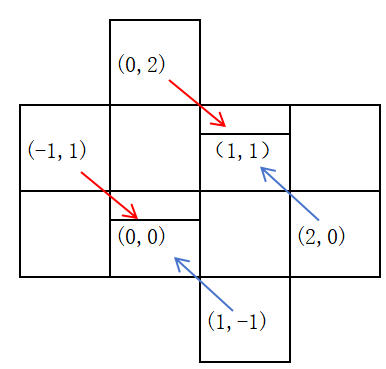}
    \caption{Step 2}\label{GM2}
\end{figure}

\item[Step 3] (Figure \ref{GM3}) Move the cells at $(-1,0)$ to the position at $(1,2)$, and one at $(2,1)$ to the position at $(0,-1)$. Then, the right orthogonal of the mutated chessboard becomes 
$$\RR_{\shO_{R}(-1,0)}\LL_{\shO_{R}(0,-1)}D_1:=D_2.$$
  \begin{figure}[ht!]
    \centering
    \includegraphics[width=0.5\linewidth]{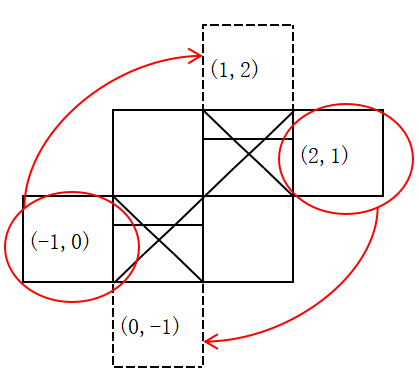}
    \caption{Step 3}\label{GM3}
\end{figure}

\item[Step 4] Consider SOD of $^{\perp}D(Y')$ in (\ref{sod2}) represented by chessboard in Figure \ref{Y'}, where the rightmost cell is attached by $(1,3)$. By symmetry of SOD (\ref{sod1}) and (\ref{sod2}), we apply Step 1 and Step 2 to the chessboard in Figure \ref{Y'}.
\begin{figure}[h!]
    \centering
    \includegraphics[width=0.25\linewidth]{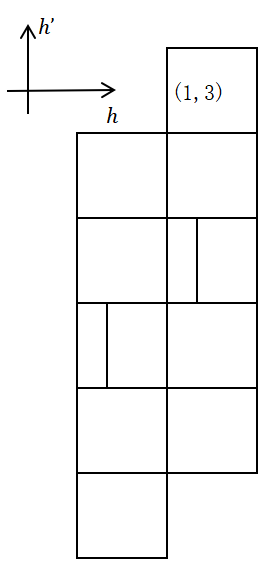}
    \caption{SOD of $^{\perp} D(Y')$}\label{Y'}
\end{figure}
The process is sketched in Figure \ref{GM4-5}, which produces new cells at $(0,0)$ and $(1,1)$:
\[
\includegraphics[width=0.05\linewidth]{attachments/B1dual.png}:=\shB=\langle \shO_{R}(-1,1), \shO_{R}, \shS', \shO_{R}(1,-1)\rangle.
\]
\begin{figure*}[h!]
    \centering
    \begin{subfigure}[t]{0.5\textwidth}
        \centering
        \includegraphics[height=2in]{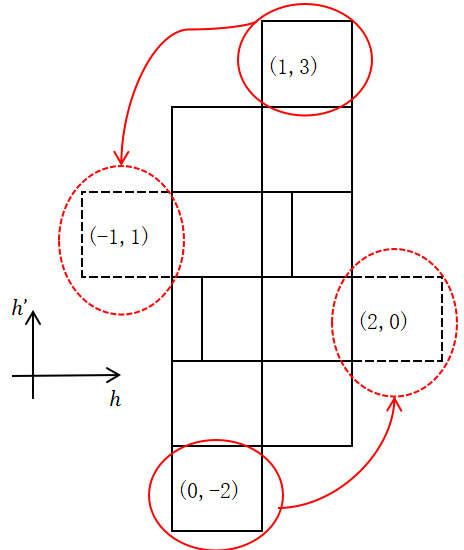}
    \end{subfigure}%
    ~ 
    \begin{subfigure}[t]{0.5\textwidth}
        \centering
        \includegraphics[height=2in]{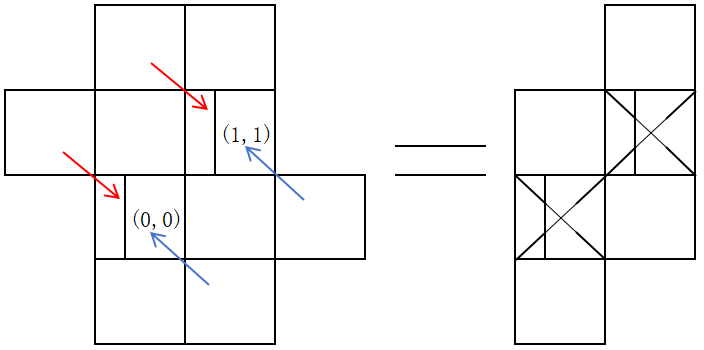}
    \end{subfigure}
    \caption{Mutation of $^{\perp}D(Y')$}\label{GM4-5}
\end{figure*}
\item[Step 5] By using Lemma \ref{van2}, both $\shA$ and $\shB$ can be identified with
$$\langle \shO_{R}, \shE, \shO_{R}(-1,1),\shO_{R}(1,-1)\rangle.$$ 
Hence, there is an equivalence $D_2\cong D'$, which implies that there is an equivalence
$$\phi: D(Y)\cong D(Y'),$$
where
$\phi(-)=\pi'_*\circ \RR_{\shO_{R}(-1,1)}\LL_{\shO_{R}(0,-2)}\circ \RR_{\shO_{R}(-1,0)}\LL_{\shO_{R}(0,-1)}\circ \RR_{\shO_{R}(-2,0)}\LL_{\shO_{R}(1,-1)}\pi^*(-).$
\end{description}

\section{Relative $D_4$ and $G_2^{\dagger}$ Flops}\label{rel}
\subsection{Relative $D_4$ flop over a base $B$} Let $\shP$ be a $spin(8)$-torsor, or a principal $spin(8)$-bundle over a smooth projective variety $B$. There is a relative $D_4$ roof over $B$, and when $B$ is a single point, diagram (\ref{relD4roof}) is (\ref{D4roof}).  
\begin{equation}\label{relD4roof}
\begin{tikzcd}
    & \shP\times_{spin(8)} OGr(3, V_8)\arrow[dl,swap, "p_+"]\arrow[dr,"p_-"]\\
     \shP\times_{spin(8)} \spinor_+ \arrow[dr, "s_+"]&&  \shP\times_{spin(8)} \spinor_-. \arrow[dl,"s_-"]\\
     & B
\end{tikzcd}
\end{equation}
Let $\shW:=\shP\times_{spin(8)}V_8$ be the  corresponding quadratic bundle over $B$ with structure group $spin(8)$. Then $\shP\times_{spin(8)} OGr(3, V_8)$ is the relative orthogonal Grassmannian $OGr(3, \shW)\rightarrow B$, and $\shP\times_{spin(8)} \spinor_{\pm}$ are the connected components $OGr(4,\shW)_{\pm}$ of $OGr(4, \shW)$ respectively. Then there is a relative simple $D_4$ flop over a base $B$ which is adapted to the flop (\ref{D4flopdiag}) by the same construction.
\begin{center}
\begin{tikzcd}
&&OGr(3,\shW) \arrow[ddll,swap,"p_+"] \arrow[d, hook,"j"]\arrow[ddrr,"p_-"] \\
&&  \arrow[dl,swap,"\pi_+"]  \shX \arrow[dr,"\pi_-"]\\
   OGr(4,\shW)_+ \arrow[drr]\arrow[r,hook,"i_+"]& \shX_+ \arrow[rr,dashed]    & & \shX_- & \arrow[l, hook',swap, "i_-"] \arrow[dll] OGr(4,\shW)_-.\\
&& B
\end{tikzcd}
\end{center}
Here, $\shX_{\pm}=\Tot_{OGr(4,\shW){\pm}}(\shU_{\pm}^{\vee}(-2h_{\pm}))$, where $\shU_{\pm}$ and $h_{\pm}$ are relative tautological rank four bundles and relative ample divisors on $OGr(4,\shW)_{\pm}$ respectively.

Since $\shX_+$ and $\shX_-$ are isomorphic after blowing up, the Orlov's formula applies.
\begin{equation}\label{relX}
    D(\shX)=\langle D(\shX_{\pm}), D(OGr(4, \shW)_{\pm}), D(OGr(4, \shW)_{\pm})(h_{\mp}), D(OGr(4, \shW)_{\pm})(2h_{\mp})\rangle
\end{equation}
By applying \cite[Theorem 3.1]{samokhin2007some} to the locally trivial fiberation $s_{\pm}:\,OGr(4, \shW)_{\pm}\longrightarrow B$, there are SODs on $D(OGr(4, \shW)_{\pm})$, represented by the same diagram (\ref{S4}) for SOD of $D(\spinor_{\pm})$. Notice that the cell \includegraphics[width=0.035\linewidth]{attachments/O.png} represents the subcategory $s_{\pm}^*(D(B))$
and the cell \includegraphics[width=0.035\linewidth]{attachments/OS.png} represents the subcategory $\langle s_{\pm}^*(D(B)),s_{\pm}^*(D(B))\otimes\shU_+^{\vee}\rangle$. 

We also need a modification of the mutations between the subcategories $s_{\pm}^*D(B)\otimes E_i$'s. The following lemma indicates that we have vanishing of the $\RR Hom$ and mutations between them if we have vanishing and mutations for every fiber.  
\begin{lemma}\label{relvan}
  Let $s:\,\shX\longrightarrow B$ be a flat morphism. Suppose there is a distinguished triangle in $D(\shX)$
  $$E_1\longrightarrow E_2\longrightarrow E_3\xrightarrow{[1]}$$
  such that for any closed point $b\in B$, 
  \begin{enumerate}
      \item[(1)] the restrictions $E_1|_b, E_2|_b, E_3|_b$ are all exceptional in $\shX_b$; 
      \item[(2)] $\RR\Hom_{\shX_b}(E_1|_b, E_3|_b)=0$. 
  \end{enumerate}
  Then 
  \begin{enumerate}
      \item[(i)] $\RR\Hom(s^*D(B)\otimes E_1, s^*D(B)\otimes E_3)=0$, and
      \item[(ii)] $\LL_{s^*D(B)\otimes E_1}(s^*D(B)\otimes E_2)=s^*D(B)\otimes E_3$.
  \end{enumerate}
\end{lemma}
\begin{proof}
Condition (1) implies that $s^*D(B)\otimes E_i$ (i=1,2,3) are admissible subcategories of $D(\shX)$ by  \cite[Theorem 3.1]{samokhin2007some}. Note that $\RR\Hom(s^*D(B)\otimes E_1, s^*D(B)\otimes E_3)=0$ holds if and only if 
$\RR \Hom(s^*D(B), E_1^{\vee}\otimes E_3)
=\RR\Hom_B(D(B), s_*(E_1^{\vee}\otimes E_3))=0$ holds by adjuctions. Let $i_b: \{b\}\hookrightarrow B$ be the closed immersion. Then, the flat base-change formula implies that
$i_b^*s_*(E_1^{\vee}\otimes E_3)=\RR \Hom_{\shX_b}(E_1|_b, E_3|_b)$, which vanishes by condition (2). Then $s_*(E_1^{\vee}\otimes E_3)$ is quasi-isomorphic to 0, and (i) holds. (ii)  follows by definition of the left mutation (\cite[Definition 2.5]{kuznetsov2007homological}). 
\end{proof}

\subsection{Proof of Theorem \ref{relthm}- relative $D_4$ case}
Firstly, we can still represent SOD (\ref{relX}) using the chessboards in Figure \ref{X+} and Figure\ref{X-}. 
Note that the restriction of the canonical divisor $\omega_{\shX}|_{OGr(3, \shW)}=3h_++3h_-+\omega_B$, and thus the Serre functor of $D(\shX)$ applies identically to the chessboard as in Section \ref{pfD4}. 
Secondly, all vanishing statements in Lemma \ref{van} also hold. This can be done by checking the vanishing for each fiber $\shX_b$ and then applying Lemma \ref{relvan}. Thus, we can exchange cells as in Section \ref{pfD4}.  
Thirdly, the relative mutations hold also by applying Lemma \ref{relvan}. 
Therefore, the sequences of mutations can be applied verbatim to SOD (\ref{relX}), and the equivalence $D(\shX_+)\cong D(\shX_-)$ follows.  
\subsection{Relative $G_2^{\dagger}$ flop over base $B$} Similarly, given a $G_2$-torsor $\shP'$ over a smooth projective variety $B$, there is a relative $G_2^{\dagger}$ roof over $B$. Recall that $R$ is the non-homogeneous $G_2$-variety in the $G_2^{\dagger}$ roof (\ref{G2+}). Let $\shQ$ and $\shQ'$ denote the locally trivial fiberations $\shP'\times_{G_2} Q_5$ and $\shP'\times_{G_2} Q'_5$ respectively.  
  
\begin{equation}\label{relG2roof}
\begin{tikzcd}
    & \shP'\times_{G_2} R
    \arrow[dl]\arrow[dr]\\
    \shQ:= \shP'\times_{G_2} Q_5 \arrow[dr]&&  \shP'\times_{G_2} Q'_5:=\shQ'. \arrow[dl]\\
     & B
\end{tikzcd}
\end{equation}
Then, the relative $G_2^{\dagger}$ flop over a base $B$ can be summarized in the following. 
\begin{center}
\begin{tikzcd}
&&\shP'\times_{G_2} R \arrow[ddll] \arrow[d, hook]\arrow[ddrr] \\
&&  \arrow[dl]  \tilde{\shY} \arrow[dr]\\
 \shQ:= \shP'\times_{G_2} Q_5\arrow[drr]\arrow[r,hook]& \shY \arrow[rr,dashed]    & & \shY' & \arrow[l, hook] \arrow[dll] \shP'\times_{G_2} Q'_5:=\shQ'.\\
&& B
\end{tikzcd}
\end{center}
\subsection{Proof of Theorem \ref{relthm}- relative $G_2^{\dagger}$ case}
 The sequences of mutations in Section \ref{G2} can be applied verbatim to SOD of $D(\tilde{\shY})$ by Lemma \ref{relvan} and the equivalence $D(\shY)\cong D(\shY')$ follows.
\section{K3 Surfaces Fiberations from $D_4$ and $G_2^{\dagger}$ roof}\label{K312}
\subsection{Cayley trick}
Let $F$ be a vector bundle of rank $r$ on a smooth projective variety $A$ and $Z$ be the zero locus of a regular section $s\in H^0(A, F)$. Then $Z$ corresponds to a regular section $\xi_s$ of $\PP(F^{\vee})$ associated with the Serre line bundle $\shO_{\PP(F^{\vee})}(1)$ (the dual of the tautological line bundle) on it. Denote $H$ the zero locus of $\xi_s$, and the geometry of these varieties can be summarized by
\begin{equation}
\begin{tikzcd}
\PP(F^{\vee}|_Z) \arrow[r, hook, "\alpha"] \arrow[d,"\pi"] &H \arrow[r,hook,"i"] \arrow[d,"q"]& \PP(F^{\vee})\arrow[dl,"p"]\\
 Z \arrow[r,hook] &A. &
\end{tikzcd}
\end{equation}
Note that $q$ is a generic $\PP^{r-2}$ fibration on $B$ whose fiber over each closed point of $Z$ is $\PP^{r-1}$ and fiber over the one outside $Z$ is $\PP^{r-2}$. We may call $Z$ the \emph{jumping loci} of $q$. The relationship between the derived categories of $H$ and $Z$ is revealed in \cite{orlov2004triangulated},
\begin{align}\label{OrlovII}
    D(H)=\langle \alpha_*\pi^*D(Z)\otimes\shO_H(-1), i^*p^*D(B), i^*(p^*D(A)\otimes\shO_E(1)), \cdots, i^*(p^*D(A)\otimes \shO_E(r-2))\rangle. 
\end{align}
\subsection{Pairs of K3 fibrations}
\subsubsection{Pair of K3 surfaces from $D_4$ roof} For simplicity, we consider the $D_4$ roof (\ref{D4roof}) over a single point first. Restrict the $D_4$ roof diagram (\ref{D4roof}) to the zero locus $M_{D_4}$ of a generic hyperplane section $\xi_{D_4}$ of $\shO(h_++h_-)$ on $OG(3, V_8)$: 
\begin{equation}
\begin{tikzcd}
&M_{D_4} \arrow[ddl,swap,"\tau_+"] \arrow[d, hook,"i"]\arrow[ddr,"\tau_-"] \\
&  \arrow[dl,"p_+"]  OGr(3, V_8) \arrow[dr,swap,"p_-"]\\
  M_+ \subset\spinor_+    & & \spinor_- \supset M_- .
\end{tikzcd}
\end{equation}
By Cayley trick, the jumping loci of $\tau_{\pm}$ is the zero locus $M_{\pm}$ of $p_{\pm*}(\xi_{D_4})\in H^0(\spinor_{\pm}, \shU_{\pm}(2h_{\pm}))$, which is a K3 surface of degree 12.
Orlov's formula (\ref{OrlovII}) implies that there are two SODs on $D(M_{D_4})$:
\begin{align}
D(M_{D_4})&=\langle D(M_+), i_*p_+^*D(\spinor_+), i_*p_+^*D(\spinor_+)(h_-), i_*p_+^*D(\spinor_+)(2h_-)\rangle\label{sodM+}\\
&=\langle D(M_-), i_*p_-^*D(\spinor_-), i_*p_-^*D(\spinor_-)(h_+),i_*p_-^*D(\spinor_-)(2h_+)\rangle.\label{sodM-}
\end{align}
which look very similar to SOD (\ref{sod+}) and (\ref{sod-}). The adjunction triangle for $i$ is 
$$ -\otimes \shO(-h_+-h_-)\rightarrow id\rightarrow i_*i^*\xrightarrow{[1]},$$
which induces a distinguished triangle on $D(E)=D(OGr(3, V_8))$
\begin{align} \label{dis2}
\RR\Hom_E(A, B(1,1))\rightarrow \RR\Hom_E(A, B)\rightarrow\RR\Hom_E(A, i_*i^*B)\xrightarrow{[1]}
\end{align}
for any $A, B\in D(E)$. 
By comparing (\ref{dis}) and (\ref{dis2}), similar statements in Lemma \ref{van} hold for $M_{D_4}$ by exactly the same calculations. Therefore the mutation steps in Section \ref{pfD4} can be applied to SOD (\ref{sodM+}) and (\ref{sodM-}) directly. Hence, we obtain a derived equivalence:
$$D(M_+)\cong D(M_-).$$
\subsubsection{Pair of K3 surfaces from $G_2^{\dagger}$ roof}
 Similarly, consider the zero locus $M_{G_2}$ of a generic section $\xi_{G_2^{\dagger}}$ of $\shO(h+h')$ on $R$ in $G_2^{\dagger}$ roof diagram (\ref{G2+}). Then the zero locus $N\subset Q$ of $p_*(\xi_{G_2^{\dagger}})\in H^0(Q, \shG^{\vee}(h))$  is a K3 surface of degree 12, and so is $N'\subset Q'$. 
\begin{equation}
\begin{tikzcd}
&M_{G_2} \arrow[ddl] \arrow[d, hook]\arrow[ddr] \\
&  \arrow[dl]  R \arrow[dr]\\
  N \subset Q    & & Q' \supset N' .
\end{tikzcd}
\end{equation}
Then proof for the derived equivalence of $G_2^{\dagger}$ flop will also give a derived equivalence:
$$D(N)\cong D(N').$$ 
\begin{remark}
Sequence (\ref{Ott+-}) and triality of $D_4$ imply that the K3 surface $N$ ( $N'$ resp.) is a degeneration of the K3 surface $M_+$($M_-$ resp.).   
\end{remark}

\subsubsection{Proof of Theorem \ref{K3}}
The relative $D_4$ roof (\ref{relD4roof}) will give two K3 fibrations $\shM_+. \shM_-$ over the base $B$ by considering the zero locus $\shM_{D_4}$ of a generic $(1,1)$ section of the roof. Just as for the relative $D_4$ flop over $B$, the Cayley trick gives two SODs for $D(\shM_{D_4})$ which are just relative versions of SOD (\ref{sodM+}) and (\ref{sodM-}). We can apply Lemma \ref{relvan} to exchange and mutate the SOD components just as in the relative $D_4$ flop case. Thus, the equivalence $D(\shM_+)\cong D(\shM_-)$ follows from the same mutation functors as in the $D_4$ flop case. The derived equivalence of the K3 fibrations $\shN, \shN'$ follows similarly.  
\bibliographystyle{alpha}
\bibliography{referencesxy}

\begin{thebibliography}{IMOU20}

\bibitem[BO02]{bondal2002derived}
Alexei Bondal and Dmitri Orlov.
\newblock Derived categories of coherent sheaves.
\newblock In {\em International Congress of Mathematicians}, page~47, 2002.

\bibitem[Har22]{hara2022derived}
Wahei Hara.
\newblock On derived equivalence for {A}buaf flop: mutation of non-commutative crepant resolutions and spherical twists.
\newblock {\em Le Matematiche}, 77(2):329--371, 2022.

\bibitem[Har24]{hara2024derived}
Wahei Hara.
\newblock Derived equivalence for the simple flop of type ${G}_2^{\dagger}$ via tilting bundles.
\newblock {\em arXiv preprint arXiv:2412.14314}, 2024.

\bibitem[IMOU20]{ito2020derived}
Atsushi Ito, Makoto Miura, Shinnosuke Okawa, and Kazushi Ueda.
\newblock Derived equivalence and {G}rothendieck ring of varieties: the case of {K}3 surfaces of degree 12 and abelian varieties.
\newblock {\em Selecta Mathematica}, 26(3):38, 2020.

\bibitem[JLX21]{jiang2021categorical}
Qingyuan Jiang, Naichung~Conan Leung, and Ying Xie.
\newblock Categorical {P}l{\"u}cker formula and homological projective duality.
\newblock {\em Journal of the European Mathematical Society}, 23(6):1859--1898, 2021.

\bibitem[Kan19]{kanemitsu2019extremal}
Akihiro Kanemitsu.
\newblock Extremal rays and nefness of tangent bundles.
\newblock {\em Michigan Mathematical Journal}, 68(2):301--322, 2019.

\bibitem[Kan22]{kanemitsu2022mukai}
Akihiro Kanemitsu.
\newblock Mukai pairs and simple {K}-equivalence.
\newblock {\em Mathematische Zeitschrift}, pages 1--21, 2022.

\bibitem[Kap88]{kapranov1988derived}
Mikhail~M Kapranov.
\newblock On the derived categories of coherent sheaves on some homogeneous spaces.
\newblock {\em Inventiones mathematicae}, 92(3):479--508, 1988.

\bibitem[Kaw02]{kawamata2002d}
Yujiro Kawamata.
\newblock D-equivalence and {K}-equivalence.
\newblock {\em Journal of Differential Geometry}, 61(1):147--171, 2002.

\bibitem[KR22]{kapustka2022mukai}
Micha{\l} Kapustka and Marco Rampazzo.
\newblock Mukai duality via roofs of projective bundles.
\newblock {\em Bulletin of the London Mathematical Society}, 54(2):694--717, 2022.

\bibitem[Kuz07]{kuznetsov2007homological}
Alexander Kuznetsov.
\newblock Homological projective duality.
\newblock {\em Publications Math{\'e}matiques de L'IH{\'E}S}, 105(1):157--220, 2007.

\bibitem[Kuz18]{kuznetsov2018derived}
Alexander Kuznetsov.
\newblock Derived equivalence of {Ito}--{Miura}--{Okawa}--{Ueda} {Calabi}--{Yau} 3-folds.
\newblock {\em Journal of the Mathematical Society of Japan}, 70(3):1007--1013, 2018.

\bibitem[LX23]{cflip}
Naichung~Conan Leung and Ying Xie.
\newblock On derived categories of generalized {G}rassmannian flips.
\newblock {\em arXiv preprint arXiv:2309.11136, to appear in Mathematical Research Letters}, 2023.

\bibitem[Mor21]{morimura2021derived}
Hayato Morimura.
\newblock Derived equivalences for the flops of type ${C}_2$ and ${A}_4^{G}$ via mutation of semiorthogonal decomposition.
\newblock {\em Algebras and Representation Theory}, pages 1--14, 2021.

\bibitem[Orl04]{orlov2004triangulated}
Dmitri~O Orlov.
\newblock Triangulated categories of singularities and d-branes in landau-ginzburg models.
\newblock {\em Proceedings of the Steklov Institute of Mathematics}, 3:227--248, 2004.

\bibitem[Ott88]{ottaviani1988spinor}
Giorgio Ottaviani.
\newblock Spinor bundles on quadrics.
\newblock {\em Transactions of the American mathematical society}, 307(1):301--316, 1988.

\bibitem[Ott90]{ottaviani1990cayley}
Giorgio Ottaviani.
\newblock On {C}ayley bundles on the five-dimensional quadric.
\newblock {\em Boll. Un. Mat. Ital. A (7)}, 4(1):87--100, 1990.

\bibitem[RX24]{rampazzo2024derived}
Marco Rampazzo and Ying Xie.
\newblock Derived equivalence for the simple flop of type ${D}_5$.
\newblock {\em arXiv preprint arXiv:2410.20446}, 2024.

\bibitem[Sam07]{samokhin2007some}
Alexander Samokhin.
\newblock Some remarks on the derived categories of coherent sheaves on homogeneous spaces.
\newblock {\em Journal of the London Mathematical Society}, 76(1):122--134, 2007.

\bibitem[Seg16]{segal2016new}
Ed~Segal.
\newblock A new 5-fold flop and derived equivalence.
\newblock {\em Bulletin of the London Mathematical Society}, 48(3):533--538, 2016.

\bibitem[Tho18]{thomas2018notes}
Richard~P. Thomas.
\newblock Notes on homological projective duality.
\newblock {\em Proc. Sympos. Pure Math. 97}, pages 585--609, 2018.

\bibitem[Ued19]{ueda2019g_2}
Kazushi Ueda.
\newblock {$G_2$}-{G}rassmannians and derived equivalences.
\newblock {\em manuscripta mathematica}, 159(3-4):549--559, 2019.

\end{thebibliography}

\end{document}